\documentclass[11pt]{amsart}
\usepackage{amsfonts,amsmath,amsthm,amssymb,
color,amscd}

\numberwithin{equation}{section}

\newtheorem{thm}{Theorem}[section]
\newtheorem*{thm*}{Theorem}
\newtheorem{prop}[thm]{Proposition}
\newtheorem*{prop*}{Proposition}
\newtheorem{question}[thm]{Question}

\newtheorem{cor}[thm]{Corollary}

\newtheorem{defin}[thm]{Definition}

\newtheorem{lemma}[thm]{Lemma}

\newtheorem{example}[thm]{Example}
\newtheorem{remark}[thm]{Remark}
\newtheorem*{remark*}{Remark}
\newtheorem*{remarks*}{Remarks}

\newcommand{\ip}[1]{\langle #1 \rangle}

\newcommand{\restrictto}[2]{\left. #1 \right|_{#2}}
\newcommand{\ddtat}{\restrictto{\frac{d}{dt}}{t=0}}

\newtheorem*{GAC}{Generalized Alekseevsky Conjecture}

\begin{document}

\title[Homogeneous {R}icci solitons]{Homogeneous {R}icci solitons}
\author[Michael Jablonski]{Michael Jablonski
}
\thanks{
    This work was supported in part by NSF grant DMS-1105647.}
\date{April 22, 2013}

\begin{abstract} In this work, we study metrics which are both homogeneous and Ricci soliton.  If there exists a transitive solvable group of isometries on a Ricci soliton, we show that it is isometric to a solvsoliton.  Moreover, unless the manifold is flat, it is necessarily simply-connected and diffeomorphic to $\mathbb R^n$.

In the general case, we prove that homogeneous Ricci solitons must be semi-algebraic Ricci solitons in the sense that they evolve under the Ricci flow by dilation and pullback by automorphisms of the isometry group.   In the special case that there exists a transitive semi-simple group of isometries on a Ricci soliton, we show that such a space is in fact Einstein.  In the compact case, we produce new proof that Ricci solitons are necessarily Einstein.

Lastly, we characterize  solvable Lie groups which admit Ricci soliton metrics.

\end{abstract}

\maketitle

\section{Introduction}
A Ricci soliton metric $g$ on a manifold $M$ is a Riemannian metric satisfying
    $$ 
    ric_g = cg + \mathcal L_X g 
    $$
for some (complete) smooth vector field $X$ on $M$ and $c\in \mathbb R$.  Such metrics correspond to  solutions of the Ricci flow which evolve by diffeomorphism and dilation.  Our interest is in homogeneous Ricci solitons.

Homogeneous solutions to the Ricci flow have long been studied for both their relative simplicity and their appearance as limits of the flow.  
This has been explored by many authors (e.g.,  \cite{IsenbergJackson:RicciFlowLocHomogGeomClosedMflds, IsenbergJacksonLu:RicciFlowLocHomog4mfld,
Guenther-Isenberg-Knopf:LinearStabilityOfHomogRicciSolitons,
Lott:LongtimeBehavOfTypeIIIRicciFlow,
BairdDanielo:ThreeDimRicciSol,
GlickensteinPayne:RicciFlow3dimUnimodLieAlgs,
Petersen-Wylie:OnGradientRicciSolitonsWithSymmetry,
Payne:TheRicciFlowForNilmanifolds,
Lauret:RicciFlowForSimplyConnectedNilmanifolds,
Lott:DimReductionAndLongTimeBehaviorOfRicciFlow,
Lauret:RicciFlowOfHomogeneousManifoldsAndItsSolitons},  see also \cite[Chapter 1]{ChowKnopf}).   Furthermore, homogeneous Ricci soliton metrics naturally arise as a  preferred choice of metric in the absence of Einstein metrics.  For example, while nilmanifolds are unable to admit Einstein metrics \cite{Jensen:HomogEinsteinSpacesofDim4}, they  often admit Ricci soliton metrics \cite{Jablo:ModuliOfEinsteinAndNoneinstein},  Ricci solitons are measurably the closest one can be to Einstein \cite{LauretNilsoliton}, and such metrics have maximal symmetry when they exist \cite{Jablo:ConceringExistenceOfEinstein}.

The classification of homogeneous Ricci soliton spaces has been a central problem for many years, motivating much of the recent work on these spaces.  If a homogeneous Ricci soliton is shrinking or steady (i.e. $c\geq 0$),  it must be either a compact  Einstein manifold or the product of a compact  Einstein manifold with Euclidean space \cite{Naber:NoncompactShrinking4SolitonsWithNonnegativeCurvature,Petersen-Wylie:OnGradientRicciSolitonsWithSymmetry}.  Thus, excluding the classification of compact homogeneous Einstein spaces (which has long been an open problem), we are reduced to classifying and understanding  expanding homogeneous Ricci solitons (i.e. $c<0$).

To date, all known examples of expanding homogeneous Ricci solitons are isometric to algebraic Ricci soliton metrics on solvable Lie groups, defined as follows.  Given a Lie group $G$, a left-invariant metric $g$ is called an \emph{algebraic Ricci soliton} if there exists a derivation $D\in Der(\mathfrak g)$ satisfying the 
condition
	$$
	Ric = cId + D
	$$
where $Ric$ is the $(1,1)$-Ricci tensor and  $\mathfrak g = Lie~G$.  Such metrics are indeed Ricci solitons and in the special case that $G$ is solvable, these have been called solvsolitons in the literature.  We note that the definition of being an algebraic Ricci soliton is dependent on the transitive group $G$ chosen - it is possible for a Ricci soliton to be algebraic with respect to one group and not another; see Section \ref{sec: examples of non algebraic}  for examples of non-algebraic Ricci solitons on solvmanifolds.  
Recall that  a Riemannian manifold $\mathcal M = (M,g)$ is called a \emph{solvmanifold} if there exists a transitive solvable group of isometries.

\begin{thm}\label{thm: Ricci soliton solvmanifolds are isometric to solsolitons} Let $\mathcal M$ be a solvmanifold which is a 
Ricci soliton.  Then $\mathcal M$ is isometric to a solvsoliton 
and the transitive solvable group may be chosen to be completely solvable.  Moreover, unless $\mathcal M$ is flat, it is simply-connected and diffeomorphic to $\mathbb R^n$.
\end{thm}

\begin{remark} This affirmatively resolves the generalized Alekseevsky conjecture among solvmanifolds, which states the following.
\end{remark}

\begin{GAC}
Let $M = G/H$ be a homogeneous Ricci soliton with $c<0$.  Then $H$ is a maximal compact subgroup of $G$ and, consequently, $M$ is diffeomorphic to $\mathbb R^n$.
\end{GAC}

Special cases of this theorem have appeared in the literature. The above theorem has been obtained for nilpotent groups in \cite{LauretNilsoliton,LauretLafuente:OnhomogeneousRiccisolitons}.  However, we warn the reader that a nilmanifold is able to admit quotients which are not nilpotent groups, but which are still solvmanifolds.

In the special case of solvsolitons, it was previously known that only simply-connected solvable Lie groups can appear, see \cite[Remark 4.12]{Lauret:SolSolitons}.  However, there do exist solvmanifolds which are not solvable groups.

In the special case of Einstein solvmanifolds, the fact that these must be simply-connected is new.  (In the further special case of negative curvature, this result follows from \cite{Alekseevski:RiemSpacesOfNegCurv}.)

In addition to the case of solvmanifolds, we have the following reductions to Einstein case.

\begin{thm}\label{thm: Ricci soliton with transitive semi-simple implies Einstein} 
Homogeneous Ricci solitons admitting a transitive semi-simple group of isometries are necessarily Einstein.
\end{thm}

\begin{thm}\label{cor: compact homog are einstein} Compact homogeneous Ricci solitons are necessarily Einstein.
\end{thm}

Theorem \ref{cor: compact homog are einstein} 
is well-known to  experts and can be deduced from the work of Petersen-Wylie \cite{Petersen-Wylie:OnGradientRicciSolitonsWithSymmetry} where homogeneous gradient Ricci solitons are studied.  To use that work, one must apply a result of Perelman \cite[Remark 3.2]{Perelman:TheEntropyFormulaForTheRicciFlowAndItsGeometricApplications} which states that compact Ricci solitons are gradient solitons.  Our proof of this theorem does not depend on \cite{Perelman:TheEntropyFormulaForTheRicciFlowAndItsGeometricApplications} and instead applies a new structure result on homogeneous Ricci solitons, described below.

\subsection*{Homogeneous Ricci solitons}
While we do not know if a general homogeneous Ricci soliton must be isometric to an algebraic Ricci soliton, we have the following structural result.  Let $G$ be Lie group with compact subgroup $K$.  We say a $G$-invariant Ricci soliton  on $G/K$ is \emph{semi-algebraic} if it satisfies
	$$
	Ric = cId + \frac{1}{2}[D_{\mathfrak g/ \mathfrak k}  + (D_{\mathfrak g/ \mathfrak k}) ^t]
	$$
for some derivation $D\in Der(\mathfrak g)$ satisfying $D:\mathfrak k \to \mathfrak k$.   Here we are evaluating the $(1,1)$-Ricci tensor on the tangent space $T_{eK}G/K$ which has been naturally identified with $\mathfrak g / \mathfrak k$, and $D_{\mathfrak g/ \mathfrak k} :  \mathfrak g/\mathfrak k \to \mathfrak g/\mathfrak k $ is the map induced by $D$.  We note that the definition of semi-algebraic is relative to a choice of transitive group $G$.

\begin{thm}\label{thm: homog RS is semi-alg relative to isometry group} Let $(M,g)$ be a connected homogeneous Ricci soliton.  Then $(M,g)$ is a semi-algebraic Ricci soliton relative to  
$G=Isom(M,g)$.
\end{thm}

In the case that $G/K$ is simply-connected, one has a converse; namely, any metric satisfying  $Ric = cId + \frac{1}{2}[D_{\mathfrak g/ \mathfrak k}  + (D_{\mathfrak g/ \mathfrak k}) ^t]$ is in fact a Ricci soliton.    
It is still an important and open question to determine if every homogeneous Ricci soliton is, in fact, simply-connected.
Further, there are no known examples of semi-algebraic Ricci solitons which are not actually algebraic.

We finish this note with a characterization of solvable Lie groups which admit Ricci soliton metrics, see Section \ref{sec: existence questions}.  This extends the work of Lauret on solvsolitons \cite{Lauret:SolSolitons}.

\bigskip
\textit{Acknowledgements.}  It is my pleasure to thank Jorge Lauret and Dan Knopf for many helpful comments on an early version of this work.

\section{Preliminaries}
\label{sec: preliminaries}

Let $g$ be a Ricci soliton metric on a manifold $M$, that is, 
	    \begin{equation}\label{eqn: ricci soliton} ric_g = cg + \mathcal L_X g  \end{equation}
for some smooth vector field $X$ on $M$ and  $c\in \mathbb R$.  It is well-known among experts that when $(M,g)$ is complete with bounded curvature, e.g. when homogeneous, that completeness of the vector field $X$ follows from the equation above  
(see e.g.\cite[Lemma 4.3]{Kotschwar:BackwardsUniquenessRF}).  Thus, the above  condition is equivalent to $g_t = (-2ct+1)\varphi_{s(t)}^*g$ being a solution to the Ricci flow
    \begin{equation}\label{eqn: Ricci Flow}\frac{\partial}{\partial t}g = -2 ric_g\end{equation}
where $\varphi_s$ is the family of diffeomorphisms generated by the vector field $X$ which we reparameterize by $s(t)=\frac{1}{c}\ln (-2ct+1)$.

\subsection*{Algebraic and semi-algebraic Ricci solitons}
On a homogeneous space $G/K$, we extend the idea of  algebraic Ricci solitons on Lie groups by considering those metrics which evolve under the Ricci flow by `automorphisms of $G/K$'.  We make this precise below.  
In the following, $G$ denotes a Lie group with compact subgroup $K$.

Each automorphism $\Phi\in Aut(G)$ which fixes $K$ gives rise to a well-defined diffeomorphism $\phi$ of $G/K$  
defined by
    $$\phi(h\cdot p) = \Phi(h)\cdot p \quad \mbox{ for } h\in G$$
Denote by $Aut(G)^K$ the subgroup of $Aut(G)$ which stabilizes $K$.  As described above, this is naturally a subgroup of $\mathfrak{Diff}(G/K)$.

Suppose there exists a 1-parameter group $\Phi_s\in Aut(G)^K$ such that  $g_t  = (-2ct+1) \phi_{s(t)}^*g$ 
is a solution to the Ricci flow, where $\phi_s \in \mathfrak{Diff}(G/K)$ corresponds to $\Phi_s$ and   $s(t)=\frac{1}{c}\ln (-2ct+1)$.  As $\Phi_s$ is a 1-parameter group, there exists a derivation $D\in Der (\mathfrak g)$ such that $d (\Phi_s)_e = exp(sD)$,  $D:\mathfrak k \to \mathfrak k$,  and
	\begin{equation}\label{eqn:semi-alg}
	Ric = cId + \frac{1}{2}[D_{\mathfrak g/ \mathfrak k}  + (D_{\mathfrak g/ \mathfrak k}) ^t]
	\end{equation}	
where $D_{\mathfrak g/ \mathfrak k} $ is the map induced by $D$  on $\mathfrak g /\mathfrak k$.  We call a $G$-invariant Ricci soliton on $G/K$ satisfying Eqn.~\ref{eqn:semi-alg}  a \textit{$G$-semi-algebraic Ricci soliton}.  Furthermore, as in the Lie groups case, we say a $G$-invariant Ricci soliton on $G/K$ is a  \textit{$G$-algebraic Ricci soliton} if it satisfies the seemingly stronger condition
	\begin{equation}\label{eqn:algebraic RS}
	Ric = cId +  D_{\mathfrak g/ \mathfrak k}
	\end{equation}		
for some $D\in Der(\mathfrak g)$ with $D:\mathfrak k \to\mathfrak k$.  When the group $G$ is understood, we abuse notation and simply say algebraic or semi-algebraic Ricci soliton.

The definitions of algebraic and semi-algebraic Ricci soliton depend on the transitive group $G$ chosen.   In Section \ref{sec: examples of non algebraic} we give examples of  homogeneous Ricci solitons which are algebraic with respect to one transitive group, but not another.  

\begin{remark} Presently, there are no examples $G$-semi-algebraic Ricci solitons which are not $G$-algebraic.   In the special case that $G$ is solvable and $K$ is trivial, semi-algebraic is equivalent to algebraic    
(Lemma \ref{lemma: evolving by aut implies algebraic}).  
\end{remark}

It is easy to see that if $G/K$ is simply-connected, then any metric satisfying Eqn.~\ref{eqn:semi-alg} is indeed a Ricci soliton.  In the sequel, we show that every  homogeneous Ricci soliton in fact satisfies Eqn.~\ref{eqn:semi-alg}.  

\begin{remark} It is still an open and important question to determine if all expanding, homogeneous Ricci solitons are simply-connected.  For more on this question, we direct the interested reader to the forthcoming work \cite{Jablo:StronglySolvable}.
\end{remark}

\subsection*{Homogeneous solutions to the Ricci flow}
Although uniqueness of solutions to the Ricci flow is not known in general (on non-compact manifolds), much more can be said for homogeneous metrics using \cite{Chen-Zhu:UniquenessRicciFlowCompleteNoncompact} and \cite{Kotschwar:BackwardsUniquenessRF}.

\begin{lemma}\label{lemma:Isom along RF}
Consider a homogeneous Riemannian manifold $(M,g)$. 
\begin{enumerate}
    \item There exists a homogeneous solution to the Ricci flow starting at the given homogeneous metric $g$.
    \item Solutions to the Ricci flow, among homogeneous metrics, are unique and, denoting the isometry group of the solution $g_t$ by $Isom(g_t)$, we have
            $$Isom(g_t)=Isom(g)$$
        for all $t$ such that the solution $g_t$ exists.
\end{enumerate}
\end{lemma}

\begin{proof}
The first statement is well-known and we do not give a proof.  Proving (i) amounts to analyzing an ODE on the space of inner products on a single tangent space of $M$.   We refer the interested reader to \cite{Lauret:RicciFlowOfHomogeneousManifoldsAndItsSolitons} for more details.

The second statement is a special case of \cite{Chen-Zhu:UniquenessRicciFlowCompleteNoncompact} and \cite{Kotschwar:BackwardsUniquenessRF}.  These works study Ricci flow in forwards and backwards time, respectively, and the metrics of interest are those which are complete and have bounded curvature.  Clearly homogeneous metrics satisfy these conditions.

\end{proof}

We observe that if one were just considering $G$-invariant metrics, uniqueness of the flow could be determined by analyzing the associated ODE on the space of inner products on a single tangent space.  However, to get uniqueness among all homogeneous metrics, that approach is not sufficient.

In the special case of Ricci solitons, one can obtain the lemma above using only \cite{Chen-Zhu:UniquenessRicciFlowCompleteNoncompact}.  
To use  that work alone, one simply recognizes that the size of the isometry group does not change along a soliton solution to Eqn.~\ref{eqn: Ricci Flow} (which is true since these isometry groups are all conjugate).

\subsection*{$G$-invariant metrics.}
Let $(M,g)$ be a homogeneous Riemannian manifold with a transitive group of isometries $G$.  Fixing a point $p\in M$, denote the isotropy at $p$ by $K=G_p$.  We  assume that $G$ is closed in $Isom(M,g)$ so that we may naturally identify $M$ with $G/K$.  Every $G$-invariant metric on $M$ arises in the following way.

Let $\mathfrak g$ denote the Lie algebra of $G$ and fix an $Ad(K)$-invariant decomposition $\mathfrak p \oplus \mathfrak k$, where $\mathfrak k$ is the Lie algebra of $K$.  The subspace $\mathfrak p$ is naturally identified with 
$T_pM$.

The restriction of an $Ad(K)$-invariant inner product $\ip{ \ ,\ }$ on $\mathfrak g$ to $\mathfrak p$ gives rise to a $G$-invariant metric $g$ on $M\simeq G/K$ defined by
    $$g(v ,v)_{p} = \ip{X_v,X_v}_e \quad \mbox{for }v\in T_{p}M$$
where $X_v\in \mathfrak p$  is the unique vector in $\mathfrak p$ such that $v=\ddtat exp(tX_v)\cdot p$.

If one considers $\Phi\in Aut(G)^K$ with corresponding $\phi\in \mathfrak{Diff}(M)$, then $\Phi(K)=K$ implies $\Phi^*\ip{ \ ,\ } $ on $\mathfrak g$ is $Ad~K$-invariant and the restriction to the $Ad(K)$-invariant subspace $\Phi(\mathfrak p)$ corresponds precisely to $\phi^*g \mbox{ on } M $.

\subsection*{Transitive unimodular group of isometries}
Let $(M,g)$ be a homogeneous Riemannian manifold with transitive group of isometries $G$.  Fix a point $p\in M$ and let $K=G_p$ denote the isotropy subgroup at $p$.  Fix an $Ad(K)$-invariant compliment $\mathfrak p$ of $\mathfrak k$ in $\mathfrak g$.

As discussed above, the set  $\mathcal M ^G$ of $G$-invariant metrics  on $M$ is naturally identified with the set of $Ad(K)$-invariant inner products on $\mathfrak p$ and thus  
can be identified with
    $$GL(\mathfrak p)^K/O(\mathfrak p)^K$$
where $GL(\mathfrak p)^K$ denotes the matrices of $GL(\mathfrak p)$ which commute with $Ad(k)$ for all $k\in K$, and $O(\mathfrak p)^K = O(\mathfrak p) \cap GL(\mathfrak p)^K$.

As is well-known,  $\mathcal M ^G$ can be endowed with a natural Riemannian metric so that $GL(\mathfrak p)^K$ acts isometrically.  To define this metric on $\mathcal M^G$, consider a $G$-invariant metric  $Q$ on $M$.  The tangent space at $Q$ is
    $$T_Q\mathcal M^G = \{ Ad(K)-\mbox{invariant symmetric, bilinear forms on } \mathfrak p\}$$
and the natural Riemannian metric on $\mathcal M^G$ is defined by
    $$\ip{v,w}_Q = tr\ vw = \sum_i v(e_i,e_i) w(e_i,e_i)$$
where $v,w\in T_Q\mathcal M^G$ and $\{e_i\}$ is a $Q$-orthonormal basis of $T_pM$.  

\begin{lemma}\label{lemma: grad sc for unimodular} Let $(M,g)$ be a homogeneous Riemannian manifold with transitive unimodular group of isometries $G$.  Given $Q\in \mathcal M^G$, denote by $ric_Q$ and $sc(Q)$ the Ricci and scalar curvatures of $(G/K,Q)$, respectively.  The gradient of the function $sc:\mathcal M^G\to \mathbb R$ is
    $$(grad\ sc)_Q = - ric_Q$$
relative to the above Riemannian metric on $\mathcal M^G$.
\end{lemma}

For a proof of this result, see \cite[Section 3]{Heber} or \cite{Nikonorov:TheScalarCurvFuctionalAndHomogEinsteinMetricsOnLieGroups}.

\section{General Homogeneous Setting}
\label{sec: general homog setting}

\begin{thm}\label{prop: RF evolving by aut of isom}
Let $(M,g)$ be a homogeneous Ricci soliton with isometry group $G$.  Fix a point $p\in M$ and let $K=G_p$ denote the isotropy subgroup of $G$.  Then  
there exist $c(t)\in\mathbb R$ and  a 1-parameter family of automorphisms $\Phi_t\in Aut(G)^K$, with associated family of diffeomorphisms $\phi_t \subset \mathfrak{Diff}(M)$, such that the Ricci flow on $M$ starting at the Ricci soliton $g$ is  given by
            $$ g_t = c(t) \phi_t^* g$$
Furthermore, $\Phi_t$ may be chosen so that 
\begin{enumerate}
    \item 
    $\Phi_t|_{K} = id$, 
    \item  
    $\Phi_t$   
    is a 1-parameter subgroup of $Aut(G)$ (after reparameterizing $t$), hence there exists $D\in Der(g)$ such that $d\Phi_t = exp(tD) \in Aut(\mathfrak g)$.
\end{enumerate}
\end{thm}

\begin{proof}[Proof of Thm.~\ref{prop: RF evolving by aut of isom}]
Let $g$ be a homogeneous Ricci soliton satisfying $ric=cg+\mathcal L_Xg$ with corresponding solution of the Ricci flow (Eqn.~\ref{eqn: Ricci Flow}) given by $g_t = (-2ct+1)\varphi_{s(t)}^*g$, where $s(t)=\frac{1}{c}\ln (-2ct+1)$ and $\varphi_s$ is the 1-parameter group of diffeomorphisms generated by $X$.

Fix a point $p\in M$.  We may assume $\varphi_t$ fixes $p$.  To see this, consider the smooth curve $\varphi_t(p)\subset M$.  By taking a section of the right $K$-action on $G$, there exists a smooth curve $h(t)\in G$ such that $L_{h(t)}\cdot p =\varphi_t(p)\subset M\simeq G/K$.

As left-translation by elements of $G$ is an isometry of $(M,g_t)$ for all $t$ (Lemma \ref{lemma:Isom along RF}), we have a solution to the Ricci flow given by
    $$g_t = (-2ct+1) \psi_{s(t)}^*g$$
where $\psi_s = L_{h(s)^{-1}} \circ \varphi_s$.  (It is not clear that this new family is a 1-parameter group of diffeomorphisms of $M$, even after reparameterizing time, as $L_{h(s)^{-1}}$ and $\varphi_s$ may not commute.)

Given $\sigma\in Isom(g)=G$, we have  $\psi_{s(t)} \circ \sigma \circ \psi_{s(t)}^{-1} \in Isom(g_t)=G$ due to Lemma \ref{lemma:Isom along RF}.  Thus, we have a smooth family of automorphisms $\Phi_s\in Aut(G)$ defined by
    $$\Phi_s(\sigma) = \psi_s \circ \sigma \circ \psi_s^{-1} \quad for \ \sigma\in G$$
Let $K=G_p$.  Since $\psi_s$ fixes $p$, we see that $\Phi_s(K)=K$, i.e. $\Phi_s\in Aut(G)^K$.  Let $\phi_s \in \mathfrak{Diff}(M)$ be the diffeomorphism of $M$ associated to $\Phi_s$.  Observe that $\phi_s = \psi_s$.  This proves the first claim, we prove (ii) next.

Thus far, we have shown our solution to the Ricci flow satisfies
    $$g_t \subset \mathbb R_{>0}\times Aut(G)^K \cdot g$$
The set of metrics $\mathbb R_{>0}\times Aut(G)^K \cdot g$ is a homogeneous manifold and so its tangent space is generated by 1-parameter subgroups of $\mathbb R_{>0}\times Aut(G)^K$.  More precisely, there exists a derivation $D\in Der(\mathfrak g)$ such that $D(\mathfrak k) \subset \mathfrak k$ (where $\mathfrak k=Lie \ K$) and
    $$ric = cg + \mathcal L_Y g $$
where $Y=\ddtat \phi_t$ generates the diffeomorphism $\phi_t \in \mathfrak{Diff}(M)$ defined by $exp(tD)\in Aut(G)^K$.  It is clear that $Y$ is a complete vector field as 1-parameter subgroups of a Lie group are defined for all time.  We obtain the desired solution to the Ricci flow by reparameterizing $\Phi_s = exp(sD)$ by $s(t)=\frac{1}{c}\ln (-2ct+1)$.  
This proves (ii).

We finish the proof of Thm.~\ref{prop: RF evolving by aut of isom} by showing $\Phi_s$ may be chosen so that $\Phi_s|_{K}=id$.  Observe that $\Phi_s|_K \in Aut(K)$, since $\Phi_s\in Aut(G)^K$.  Note that the subgroup $K=G_p$ is compact since $G$ is the isometry group of $(M,g)$.

Denote the connected component of the identity of $K$ by $K_0$.  Since $K_0$ is compact and connected, it may be written as a product
    $$K_0=K_{ss}Z(K_0)$$
of its maximal semi-simple subgroup $K_{ss}=[K_0,K_0]$ and its center $Z(K_0)$.  (The intersection $K_{ss}\cap Z(K_0)$ is finite.)

The center of a group is always preserved by automorphisms and so $\Phi_s|_{Z(K_0)} \in Aut(Z(K_0))$.  Being the automorphism group of a compact torus, $Aut(Z(k_0))$ is finite and since $\Phi_0=id$, we see that $\Phi_s|_{Z(K_0)}=id$.

Since $\Phi_s(K_{ss})$ is semi-simple and $K_{ss}$ is a maximal semi-simple subgroup of $K_0$, we see that $\Phi_s|_{K_{ss}}\in Aut(K_{ss})$.  The group $K_{ss}$ being semi-simple implies the connected component of the identity of $Aut(K_{ss})$ consists of inner automorphisms and so $\Phi_s|_{K_{ss}} = C_{exp(s X)}$ for some $X\in Lie\ K_{ss}$, where $C_h$ denotes conjugation by $h\in K_{ss}$.

As conjugation is the identity on the center, we see that
    $$\Phi_s|_{K_0} = C_{exp(sX)}$$
for some $X\in Lie \ K_{ss}$.

Consider the composition $\Phi_s \circ C_{exp(sX)^{-1}}$.  By construction, this is the identity on $K_0$.  Since $\Phi_s$ and $C_{exp(s X)}$ commute, this composition is a 1-parameter subgroup of $Aut(G)^K$.  Furthermore, this gives rise to a continuous family of automorphisms of the finite group $K/K_0$.  As the automorphism group of a finite group is finite, $\Phi_s \circ C_{exp(sX)^{-1}}$ is the identity on $K$ as it is so for $s=0$.

Lastly, $C_{exp(sX)}\in Aut(G)^K$ corresponds to the diffeomorphism of $M$ which is left-translation by $exp(sX)\in K$.  As this is an isometry, the diffeomorphisms of $M$ corresponding to the family $\Phi_{s(t)} \circ C_{exp(s(t)X)^{-1}}$  give rise to a solution to the Ricci flow satisfying (i).  This completes the proof of the theorem.

\end{proof}

Theorem \ref{thm: homog RS is semi-alg relative to isometry group} now follows immediately from the theorem above; see also the discussion surrounding Eqn.~\ref{eqn:semi-alg}.

\section{Semi-simple and compact homogeneous spaces}
\label{sec: semi-simple}

This section is devoted to proving Theorems \ref{thm: Ricci soliton with transitive semi-simple implies Einstein} and  \ref{cor: compact homog are einstein} which state that a homogeneous Ricci soliton is necessarily Einstein if there exists a transitive semi-simple group of isometries or it is compact.

Before presenting the proofs of these results, we note the following corollary to Theorem \ref{thm: Ricci soliton with transitive semi-simple implies Einstein}.

\begin{cor} On the group $SL_2\mathbb R$, left-invariant Ricci soliton metrics do not exist. 
\end{cor}

Although special cases of this corollary were known to experts, to our knowledge, this result is new. We remind the reader that the group $SL_2\mathbb R$ is the only non-compact semi-simple group for which the existence of Einstein metrics has been completely resolved.  In this case, there are no left-invariant Einstein metrics and the corollary follows immediately from Theorem \ref{thm: Ricci soliton with transitive semi-simple implies Einstein}.

\subsection*{Proof of Thm.~\ref{thm: Ricci soliton with transitive semi-simple implies Einstein}}
As in the case of solvmanifolds, the proof of Theorem \ref{thm: Ricci soliton with transitive semi-simple implies Einstein} uses the fact that there is a very special transitive group contained in the  isometry group $G$.  We may assume our homogeneous space $M$ is connected.  The following is Theorem 4.1 of \cite{Gordon:RiemannianIsometryGroupsContainingTransitiveReductiveSubgroups}.

\begin{thm}[Gordon]\label{thm: gordon-transitive reductive} Suppose a connected Lie group $H$ with compact radical acts transitively and effectively by isometries on a Riemannian manifold $M$.  Then the connected component of the isometry group is reductive.
\end{thm}

This theorem clearly applies when the transitive group $H$ is semi-simple as the radical of a semi-simple group is trivial.

Let $G$ denote the isometry group of $(M,g)$ and $H$ a transitive semi-simple subgroup of $G$.  Let $G_0$ denote the connected component of $G$.

Recall, the maximal semi-simple subgroups of a connected Lie group are all conjugate by elements of the nilradical.  Since the nilradical of $G_0$ is contained in the center (cf. Thm.~\ref{thm: gordon-transitive reductive}), we see that $G_0$ contains a unique maximal semi-simple subgroup which is normal in $G_0$.  As $H$ is contained in this maximal, semi-simple subgroup, this maximal subgroup acts transitively on $M$ and we may replace $H$ with said maximal, semi-simple subgroup of $G_0$.  Furthermore, this maximal, semi-simple subgroup is closed and the  results from the previous section may be applied.

Let $\Phi_t$ be the automorphisms of $G$ by which the Ricci flow on $M$ is evolving (cf. Thm.~\ref{prop: RF evolving by aut of isom}).  Let $\phi_t$ denote the diffeomorphisms of $M$ (by which Ricci flow is evolving) corresponding to $\Phi_t$ which fix some point $p\in M$.  Denote by $K$ the stabilizer subgroup of $H$ at $p$.

Since $\Phi_t(H)$ is a semi-simple subgroup of $G$, we have $\Phi_t(H) = H$, by maximality.  Thus, $\Phi_t|_H \in Aut(H)$ and by Eqn.~\ref{eqn:semi-alg}  there exists $X\in \mathfrak h$ such that
    \begin{equation}\label{eqn: ric for semi-simple} Ric = cId + \frac{1}{2}[ (ad\ X)_{\mathfrak h/\mathfrak k} + (ad\ X)^t _{\mathfrak h/\mathfrak k}]\end{equation}
Here we have used the fact that the derivations of a semi-simple Lie algebra are all inner. To finish the proof of the theorem, we present a lemma whose proof we postpone until the end of the current proof.

\begin{lemma}\label{lemma: trace of pr circ ad X is zero} Let $X\in\mathfrak h$ be as above.  Then $tr ( ad\ X)_{\mathfrak h/\mathfrak k} =0 $.
\end{lemma}

As scalar curvature is a Riemannian invariant,  $sc(g) = sc(\phi_t^*g)$.  Using Lemma \ref{lemma: grad sc for unimodular} together with Eqn.~\ref{eqn: ric for semi-simple} and Lemma \ref{lemma: trace of pr circ ad X is zero}, we have
    \begin{eqnarray*}0 = \ddtat sc(\phi_t^*g) &=& \ip{grad\ sc,  (ad\ X)_{\mathfrak h/\mathfrak k}}_g \\
        &=& c tr( ad\ X)_{\mathfrak h/\mathfrak k} + tr\ \frac{1}{2}[ ( ad\ X)_{\mathfrak h/\mathfrak k} + ( ad\ X)_{\mathfrak h/\mathfrak k}^t ] ( ad\ X)_{\mathfrak h/\mathfrak k} \\
        &=& tr\frac{1}{4}[  (ad\ X)_{\mathfrak h/\mathfrak k} + ( ad\ X)^t_{\mathfrak h/\mathfrak k} ]^2
    \end{eqnarray*}
This implies $\frac{1}{2}[ (ad\ X)_{\mathfrak h/\mathfrak k} + (ad\ X)^t_{\mathfrak h/\mathfrak k} ]=0$ and hence the metric is Einstein.

\medskip

We finish by proving the last lemma.

\begin{proof}[Proof of Lemma \ref{lemma: trace of pr circ ad X is zero}] 
Let $K$ denote the isotropy of the $H$-action on $M$ at $p$.  Let $\mathfrak p$ denote an $Ad(K)$-invariant compliment of $\mathfrak k$ in $\mathfrak h$, as above.  Recall, $K$ is invariant under the 1-parameter family of automorphisms $\Phi_t = exp(t \ ad\ X)$ of $H$ and we have
    $$tr\ ad\ X = tr (pr \circ ad\ X) + tr\ ad\ X|_{\mathfrak k}$$
where $pr:\mathfrak h \to \mathfrak p$ is the projection onto $\mathfrak p$, as before.

However, using Thm.~\ref{prop: RF evolving by aut of isom} (i), we may assume that $ad\ X|_\mathfrak k = 0$.  Since $\mathfrak h$ is semi-simple, we have $0=tr\ ad\ X = tr (pr \circ ad\ X) = tr (ad\ X)_{\mathfrak h/\mathfrak k}$, which proves the lemma.
\end{proof}

\begin{remark*}  In the special case of a left-invariant metric on a semi-simple group satisfying the stronger condition $Ric = cId +D$, this result was proven by J. Lauret, see Thm.~5.1 of \cite{LauretNilsoliton}.  The proof given there inspired the general case proven here.
\end{remark*}

Next we prove Theorem \ref{cor: compact homog are einstein}.  While this can be deduced from Theorem \ref{thm: Ricci soliton with transitive semi-simple implies Einstein} using topological properties of compact Ricci solitons, there is a direct algebraic proof which is shorter.  This is the proof we give.

\begin{proof}[Proof of Theorem \ref{cor: compact homog are einstein}]
Consider a compact homogeneous Ricci soliton $(M,g)$.  As $M$ is compact, the isometry group $G$ of $(M,g)$ is also compact.

Let $\Phi_t\in Aut(G)$ be the family of automorphisms of $G$ which induces the soliton metric (see Thm.~\ref{prop: RF evolving by aut of isom}).  Since $G$ is compact, we see that there exists $X\in [\mathfrak g,\mathfrak g]$ such that $\Phi_t|_{G_0} = exp(t\ ad \ X)$ (see the proof of Thm.~\ref{prop: RF evolving by aut of isom} (i)), where $G_0$ is the identity component of $G$.

As compact groups are unimodular, we may use Lemma \ref{lemma: grad sc for unimodular} and finish the proof of the theorem as in the proof of Theorem \ref{thm: Ricci soliton with transitive semi-simple implies Einstein}.
\end{proof}

\begin{remark} While we expect that any homogeneous Ricci soliton admitting a reductive group of isometries must be Einstein, this is still an open question.
\end{remark}

\section{Solvmanifolds}
\label{sec: main result}

We prove Theorem \ref{thm: Ricci soliton solvmanifolds are isometric to solsolitons} first assuming $\mathcal M = \widetilde{\mathcal M}$ is simply-connected.

\subsection*{A canonical presentation}
To study solvmanifolds, we first produce a preferred transitive solvable group of isometries which is in so-called  \textit{standard position}.

\begin{defin} Let $G$ be a Lie group and $S$ a subgroup of $G$.  We say that $S$ is in standard position in $G$ if, among Lie subgroups of $G$ containing $S$,  the normalizer $N_G(S)$ is maximal with respect to the property that it contains no non-trivial noncompact semi-simple subgroups.  When $G$ is understood, we will just say $S$ is in standard position.
\end{defin}

Let $\mathcal M = (M,g)$ be a solvmanifold.  Up to conjugacy within the isometry group, there is a unique solvable Lie group $S < Isom(\mathcal M)_0$ which acts almost simply transitively and is in standard position in $Isom(\mathcal M)_0$, see Section 1 and Definition 1.10 of \cite{GordonWilson:IsomGrpsOfRiemSolv}.

The definition of $S$ being in standard position in $G$ is purely Lie theoretic, depending only on the embedding of $S$ in $G$.  Consequently, by Lemma \ref{lemma:Isom along RF}, we have the following.

\begin{lemma} Let $\mathcal M$ be a solvmanifold and $\mathcal M_t$ a homogeneous solution to the Ricci flow.  If $S$ is in standard position for $\mathcal M$, then it is in standard position for $\mathcal M_t$.
\end{lemma}

As we are assuming $\mathcal M$ is simply connected, $S$ acts simply transitively on $\mathcal M$ and so $\mathcal M \simeq \{S,g\}$ where $g$ is a left-invariant metric on $S$.  The solvmanifold $\mathcal M_t$ can be written as $\{S,g_t\}$, where $g_t$ is a left-invariant metric on $S$.  Now let  $\mathcal M$ be a Ricci soliton with metric $g_0$ and consider the homogeneous solution $g_t$ to the Ricci flow given above:  $g_t = (-2ct+1) \varphi_{s(t)} ^* g_0$ where $s(t)=\frac{1}{c}\ln (-2ct+1)$ and $\varphi_s \in \mathfrak{Diff}(M)$ is generated by $X$.  We consider the dilation $h_t=\frac{1}{-2ct+1} \ g_t = \varphi_{s(t)}^* g_0$.

As dilation does not change the isometry group, $S$ is in standard position for both $g_0$ and $h_t$, and we may apply Theorem 5.2 of \cite{GordonWilson:IsomGrpsOfRiemSolv} to see that there exists a 1-parameter family $\Psi_t \in Aut(S)$, each of which is an isometry $\Psi_t : \{S,g_0\} \to \{S,h_t\}$; that is, the Ricci flow is evolving by automorphisms relative to $S$.

\begin{remark*} A priori, one might be concerned about the smoothness of the family $\Psi_t$. However, one can see that $h_t$ is tangent (first at $t=0$ and, hence, for all $t$) to $Aut(S)^* g_0$ and so $\Psi_t$ may be chosen to be smooth.

Moreover, as in the general homogeneous setting (Theorem \ref{prop: RF evolving by aut of isom}), after reparameterizing time we have $d(\Psi_t)_e = exp(tD) \in Aut(\mathfrak s)$ for some $D\in Der(\mathfrak s)$.
\end{remark*}

Here and through out, we are making the usual identification between the automorphism group $Aut(S)$ of a simply-connected Lie group $S$ and the automorphism group $Aut(\mathfrak s)$ of its Lie algebra $\mathfrak s$: $\Phi\in Aut(S) \leftrightarrow d\Phi_e\in Aut(\mathfrak s)$.

\medskip

Thus far, we have shown that a simply-connected Ricci soliton solvmanifold $\mathcal M$ is isometric to $\{S,g\}$ (where $S$ is in standard position) and the Ricci flow evolves by automorphisms, i.e.
    \begin{equation}\label{eqn: evolving by Aut}Ric_g = cId + \frac{1}{2}(D+D^t)\end{equation}
where $Ric_g$ is the $(1,1)$-Ricci tensor of the left-invariant metric $g$ on $S$.  Here we have evaluated the tensor on the tangent space $T_eS \simeq \mathfrak s$.

\subsection*{Semi-algebraic implies algebraic}
The second part of the proof of Theorem \ref{thm: Ricci soliton solvmanifolds are isometric to solsolitons} is to show that if the Ricci flow evolves by automorphisms, then the soliton is actually algebraic.

\begin{lemma}\label{lemma: evolving by aut implies algebraic} Let $\{S,g\}$ be a solvable Lie group with left-invariant metric.  If $Ric_g$ satisfies Eqn.~\ref{eqn: evolving by Aut}, then
    $$\frac{1}{2}(D+D^t)\in Der(\mathfrak s)$$
and hence $\{S,g\}$ is an algebraic Ricci soliton.
\end{lemma}

The proof of this lemma amounts to a careful reading of \cite{Lauret:SolSolitons}; while we do not provide full details, we make note of what modifications need to be done to that work.

We divide the proof into two cases. Given $D\in Der(\mathfrak s)$, denote by  $S(D)=\frac{1}{2}(D+D^t)$  the symmetric part of $D$, relative to $g$.

\medskip

\textbf{Case $c\geq 0$.}  The proof of this case follows the proof of \cite[Prop.~4.6]{Lauret:SolSolitons}. The analysis there with $Ric = cId+D$ works with $Ric=cId+S(D)$ upon noting that $D(\mathfrak s)\subset \mathfrak n$, the nilradical of $\mathfrak s$, for any derivation $D\in Der(\mathfrak s)$.  The conclusion of \cite[Prop.~4.6]{Lauret:SolSolitons} is that the manifold is Ricci flat, hence one may take $c=0$ and $D=0$.

\begin{remark*}
By a result of Alekseevsky-Kimel'fel'd \cite{AlekseevskiiKimelfeld:StructureOfHomogRiemSpacesWithZeroRicciCurv}, $\{S,g\}$ must be flat and hence $S$ is abelian (being in standard position).
\end{remark*}

\textbf{Case $c<0$.} The proof of this case follows the proof of \cite[Thm.~4.8]{Lauret:SolSolitons}.  The work there remains valid with the hypothesis $Ric = cId +D$ replaced with $Ric = cId+S(D)$.  The only modification that the proof needs is the term $F$ should now be $S(ad~H + D)$ instead of $S(ad~H)+D$.  Consequently, \cite[Thm~4.8]{Lauret:SolSolitons} now shows that $S(D) \in Der(\mathfrak s)$.

\medskip

So far we have shown that any simply-connected  Ricci soliton solvmanifold is isometric to an algebraic Ricci soliton relative to some, possibly different, solvable Lie structure.  The fact that $S$ may be chosen to be completely solvable is the content of \cite[Cor.~4.10]{Lauret:SolSolitons}.  We note that completely solvable groups are in standard position, see \cite[Thm.~4.3]{GordonWilson:IsomGrpsOfRiemSolv}.  This proves Theorem \ref{thm: Ricci soliton solvmanifolds are isometric to solsolitons} in the case that $\mathcal M$ is simply-connected.

\section{Simple Connectivity}
\label{sec: simple connectivity}

In this section, we show that any Ricci soliton solvmanifold is simply-connected, unless it is flat.  It is still an open question whether every homogeneous Ricci soliton must be simply-connected.  For more on this question, we direct the interested reader to the forthcoming work \cite{Jablo:StronglySolvable}.

Let $\mathcal M$ be a non-flat solvmanifold with simply-connected cover $\widetilde{\mathcal M}$.  Assume that $\mathcal M$ is a Ricci soliton.  As the projection $\pi : \widetilde{\mathcal M} \to \mathcal M$ is a local isometry, $\widetilde{\mathcal M}$ is also a Ricci soliton.  The Ricci soliton on $\widetilde{\mathcal M}$  corresponds to the family of diffeomorphisms $\widetilde{\varphi_t}$  which satisfy the following commutative diagram
    $$\begin{CD}
    \widetilde{\mathcal M} @>{\widetilde \varphi_t}>> \widetilde{\mathcal M}\\
    @VV \pi V  @VV \pi V\\
    \mathcal M @> \varphi_t >> \mathcal M
    \end{CD}  $$
where $\widetilde \varphi_t$ is the lift of $\varphi_t$.  Moreover, $\widetilde{\mathcal M}$ is 
homogeneous. 

\begin{prop} The degree of the cover is 1, hence $\mathcal M$ is simply-connected.
\end{prop}

As the proposition suggests, we will show that $\mathcal M$ is simply-connected by showing that the fiber $\pi^{-1}(p)$ consists of a single point for one, hence any, point $p\in\mathcal M$. A more technical version of this proposition is Proposition~\ref{prop: C-orbit has one point}.

\subsection*{The preimage $\pi^{-1}(p) \subset \widetilde{\mathcal M}$ as an orbit}
By definition, $\mathcal M$ admits a transitive solvable group of isometries and, in fact,  can be realized as 
	$$\mathcal M \simeq R/C$$
where $R$ is a simply-connected solvable group acting almost simply transitively and $C$ is a discrete subgroup of $\{ x\in R \ | \ Ad(x) \in O(\mathfrak r, \ip{\ ,\ }')  \}$, where $\ip{\ ,\ }'$ is the induced inner product on $\mathfrak r \simeq T_e R$  (see \cite[Lemma 1.2]{GordonWilson:IsomGrpsOfRiemSolv}).  The group $R$ acts simply transitively on $\widetilde{\mathcal M} \simeq R$.

As $R$ acts transitively on $\widetilde{\mathcal M}$, it is a `modification' of a group in standard position, which we denote by 
$S$ (see \cite[Thm.~3.1]{GordonWilson:IsomGrpsOfRiemSolv}).  To make this more precise, we fix a point $\widetilde{p} \in \widetilde{\mathcal M}$ which we identify with $e_R\in R$ and $e_S\in S$, the identity elements of $R$ and $S$, respectively.  In this way, we identify $T_{\widetilde{p}}  \widetilde{\mathcal M} \simeq \mathfrak r \simeq \mathfrak s$.  The induced inner product on $\mathfrak r$ will be denoted $\ip{\ ,\ }'$ while the induced inner product on $\mathfrak s$ will be denoted $\ip{\ ,\ }$.

The Lie algebra $\mathfrak r$ satisfies $\mathfrak r \subset \mathfrak k + \mathfrak s$, where $\mathfrak k =\mathfrak{so}(\ip{\ ,\ }) \cap Der(\mathfrak s)$ is the set of skew-symmetric derivations of $\mathfrak s$, relative to $\ip{\ ,\ }$.   Further, there exists a linear map $\phi : \mathfrak s \to \mathfrak k$ (called a modification map)  so that $r = (Id + \phi) \mathfrak s$.  From this perspective, the inner products on $\mathfrak r$ and $\mathfrak s$ are related by
\begin{equation}\label{eqn: relating inner products on r and s}
	    \ip{X + \phi(X),X+\phi(X)}' = \ip{X,X}   \quad \quad  \mbox{for } X\in \mathfrak s
\end{equation}	

As $R$ is connected, it is a subgroup of $K\ltimes S$, where $K = SO(\ip{ \ ,\ }) \cap Aut(\mathfrak s)$.  (Here we are making the usual identification between $Aut(S)$ and $Aut(\mathfrak s)$: $\Phi \in Aut(S)$ corresponds to $d\Phi_e \in Aut(\mathfrak s)$.  This is possible since $S$ is simply-connected.)  The group $R$ acts on $S$ by isometries via the usual action of $K\ltimes S$ on $S$, and we can think of the manifold $\widetilde{\mathcal M}$ as either $R$ or $S$ with the appropriate left-invariant metric.

\begin{remark} In the notation of \cite{GordonWilson:IsomGrpsOfRiemSolv}, $\mathfrak k = N_\mathfrak l (\mathfrak s)$ and $K = N_L(S)$.  For a more detailed and thorough introduction to modifications, we refer the interested reader to Sections 2 and 3 of \cite{GordonWilson:IsomGrpsOfRiemSolv}.
\end{remark}

To show that $\mathcal M$ is simply-connected, we  show that the number of elements in $\pi^{-1}(p) \subset \widetilde{\mathcal M}$ is one.  We choose $p=[C]\in R/C$, so that  $\pi^{-1}(p) = C \subset R$.  Note that $\widetilde{p}\in \pi^{-1}(p)$ and the set $C$ in $R$ is the orbit $C\cdot e_R = C\cdot \widetilde{p}$ which viewed on $S$ is the set $C\cdot \widetilde{p} = C\cdot e_S$, where $C$ is acting on $S$ as a subgroup of $K\ltimes S$.

\begin{prop}\label{prop: C-orbit has one point} $C\cdot e_S = e_S$.  Consequently,  the number of elements in $\pi^{-1}(p)$ is one and $\mathcal M$ is simply-connected.
\end{prop}

To prove the proposition, we  carefully analyze the relationship between the two different homogeneous structures on $\widetilde{\mathcal M}$ coming from $R$ and $S$.

\subsection*{The soliton diffeomorphisms}

\begin{lemma}\label{lemma: varphi_t exists for all time} Fix any point $p\in \mathcal M$.  Let $\varphi_t$ be the 1-parameter group generated by the vector field $X$ appearing in Eqn.~\ref{eqn: ricci soliton}.  We may replace $\varphi_t$ with a family which fixes $p$.  Moreover, this new family of diffeomorphisms exists for all $t$.
\end{lemma}

\begin{proof} The first claim follows from the transitivity of $R$ on $\mathcal M$. 
Take $r(\varphi_t(p)) \in R$ which maps to $\varphi_t(p) \in R/C$ under the usual quotient $R\to R/C$.
As $C$ is discrete,  $R/C$ is locally diffeomorphic to $R$ and this choice can be made smoothly in $t$.   By Lemma~\ref{lemma:Isom along RF}, the translation $L_{r(\varphi_t(p))^{-1}}$ is an isometry  of $R/C$ relative to each of the metrics $\{ g_s | s\in\mathbb R\}$.  Upon  composing  $\varphi_t$ with the isometry $L_{r(\varphi_t(p))^{-1}}$ we have that $p$ is fixed.
(Note, this new family yields a solution of the Ricci flow (Eqn.~\ref{eqn: Ricci Flow}) as Ricci flow is invariant under isometries.)

The second claim follows immediately from the definition of the vector field $X$ being complete, i.e., the family of diffeomorphisms $\varphi_t$ it generates exists for all time.  Composing $\varphi_t$ with the time dependent translations $L_{r(\varphi_t(p))^{-1}}$ does not change the time interval on which the family exists.
\end{proof}

We pick $p = [C] \in \mathcal M$.  As $\varphi_t$ fixes $p$,  $\widetilde \varphi_t$ stabilizes $\pi^{-1}(p)$.  This set is discrete and, since $\widetilde \varphi_0$ is the identity map, we see that $\widetilde \varphi_t(q)=q$ for all $q\in \pi^{-1}(p)$.

For the moment, we view $\widetilde{\mathcal M}$ as the solvable group $S$ in standard position with left-invariant metric $g$.  
As shown in Section \ref{sec: main result}, there exists a 1-parameter group of automorphisms $\Psi_s \subset Aut(S)$ such that $\widetilde{\varphi_{s(t)}}^* g = \Psi_{s(t)}^* g$ where $s(t)=\frac{1}{c}\ln (-2ct+1)$.  This equality holds true as long as our solution  to the Ricci flow exist, which is on the interval $(\frac{1}{2c},\infty)$.

\begin{lemma}\label{lemma: I_t exists for all time} Given $\widetilde{\varphi_s}$ and $\Psi_s$ as above, $\widetilde{\varphi_s}^* g = \Psi_s^* g$ for all $s\in \mathbb R$.  Consequently, there exists a family of isometries $I_s$ defined for all $s$ such that $\widetilde{\varphi_s} = \Psi_s \circ I_s$.  Moreover, $I_s$ is contained in the isotropy at $\widetilde{p}$, which is compact.
\end{lemma}

\begin{proof}
As $t$ ranges over $(\frac{1}{2c},\infty)$, $s(t)=\frac{1}{c}\ln (-2ct+1)$ ranges over all real numbers. This proves the first claim.

The equality $\widetilde{\varphi}_{s}^*g = \Psi_{s}^*g$ implies that $I_s = (\Psi_s)^{-1} \widetilde{\varphi}_s$ is an isometry.  Since $\widetilde{\varphi_s}$ and $\Psi_s$ both fix $\widetilde{p}=e_S$, we see that $I_s$ is contained in the isotropy at $\widetilde{p}$, which is compact.
\end{proof}

From the work in Section \ref{sec: main result}, we know that $S$ is completely solvable.  Combining this with $S$ being simply-connected, we have that   the Lie group exponential $\exp: \mathfrak s \to S$ is a diffeomorphism.   As such, it has an inverse, $\log$, which is also a diffeomorphism.  (Unless otherwise stated,  $\exp$ and $\log$ refer only to those for $S$.)  We study the condition that $\widetilde{\varphi_t} = \Psi_t \circ I_t$ fixes every element $c\cdot \widetilde{p} \in \pi^{-1}(p)$, where $c\in C < R$.

\begin{lemma}\label{lemma: log cp in Ker D} For $c\in C$, $\log (c\cdot \widetilde{p}) \in Ker\ D$, where $D$ is the derivation of $\mathfrak s$ which generates the family of automorphisms $d(\Psi_t)_e = exp(tD)$ of $\mathfrak s$.
\end{lemma}

\begin{proof} By taking the log, the statement that  $\widetilde{\varphi_t} = \Psi_t \circ I_t$ fixes $c\cdot \widetilde{p} \in \pi^{-1}(p) \subset S$ is equivalent to $\exp(-tD) (\log (c\cdot \widetilde{p})) = \log(I_t(c\cdot \widetilde{p})) \in \mathfrak s$.

As $I_t$ is contained in the compact isotropy group at $\widetilde{p}$, $\log(I_t(c\cdot \widetilde{p}))$ is contained in a compact set of $\mathfrak s$.
Using Lemma \ref{lemma: I_t exists for all time}, we may let $t\to \pm \infty$.  As $D$ is symmetric, $\exp(-tD)(\log (c\cdot \widetilde{p}))$ being contained in a compact set as $t\to \pm \infty$ is equivalent to $\log (c\cdot \widetilde{p}) \in Ker\ D$.
\end{proof}

\begin{remark*}From this point forward, $D$ will only refer to the derivation above, unless otherwise stated.
\end{remark*}

\subsection*{Decomposing $\mathfrak s$ and the kernel of $D$}
Relative to the solvsoliton metric $\ip{\ ,\ }$ on $\mathfrak s$, we may decompose our completely solvable algebra as
	$$\mathfrak s = \mathfrak a + \mathfrak n$$
where $\mathfrak n$ is the nilradical of $\mathfrak s$ and $\mathfrak a = \mathfrak n^\perp$.  The set $\mathfrak a$ is abelian and $ad\ A: \mathfrak n\to \mathfrak n$ is symmetric for $A\in\mathfrak a$, see \cite[Thm.~4.8]{Lauret:SolSolitons}.

Using the above decomposition, we have $D|_\mathfrak a =0$ and  $D|_\mathfrak n = D_1 - ad\ H$ where $H$ is the `mean curvature vector' satisfying $\ip{H,X} = tr\ ad\ X$, for all $X\in \mathfrak s$, and $D_1$ is the pre-Einstein derivation of $\mathfrak n$ (see \cite[Prop.~4.3]{Lauret:SolSolitons}).    As the kernel of this derivation is important, we study it more closely.

\begin{lemma}\label{lemma: Ker D in s_1} $Ker\ D \subset \mathfrak a + Im(ad\ H)$, where $Im(ad\ H)$ denotes the image of $ad\ H: \mathfrak s \to \mathfrak n$.
\end{lemma}

\begin{proof} As $\mathfrak a\subset Ker\ D$, it suffices to understand $Ker\ D \cap \mathfrak n$.  For $X\in Ker\ D \cap \mathfrak n$, we have
    $$D_1 X =ad\ H(X)$$
As $\mathfrak n$ with the induced metric is nilsoliton (\cite[Theorem 4.8]{Lauret:SolSolitons}), it is the nilradical of a solvable Lie group with Einstein metric (\cite[Prop.~4.3]{Lauret:SolSolitons}).  As such, we know that the derivation $D_1$ is positive definite on $\mathfrak n$ (\cite[Theorem 4.14]{Heber})   
and all symmetric and skew-symmetric derivations of $\mathfrak n$ commute with $D_1$ (\cite[Lemma 2.2]{Heber}).     
The derivation $ad\ H \in Der(\mathfrak n)$ is a symmetric derivation and hence commutes with $D_1$.  Decomposing $X = \sum X_\lambda$ into a sum of eigenvectors of $D_1$, with $D_1X_\lambda = \lambda X_\lambda$, we see that $ad~ H (X_\lambda) = \lambda X_\lambda$, as well.  Letting $Y =  \sum \frac{1}{\lambda}X_\lambda$, we have $X = ad\ H (Y)$.
\end{proof}

We now write $\mathfrak s = \mathfrak a \oplus \mathfrak n = \mathfrak a \oplus \mathfrak n_1 \oplus \mathfrak n_0$, where $\mathfrak n_1 = Im(ad\ H)$ 
and $\mathfrak n_0=\mathfrak n \cap Ker\ ad\ H$.  
Define $\mathfrak s_1 =\mathfrak a \oplus [\mathfrak s,\mathfrak s]$ with corresponding group $S_1 <S$.  Note that $\mathfrak n_1 \subset \mathfrak s_1$.

\begin{lemma}\label{lemma: commuting derivations} \label{lemma: n1 and s1 stable under K}
Using the notation above, we have
\begin{enumerate}
    \item The set $\mathfrak k$ of skew-symmetric derivations annihilates $\mathfrak a$ and, hence,     
    $[E,  ad~A] = 0$
    for $E\in \mathfrak k$   and $A\in \mathfrak a$,
    \item For $E\in \mathfrak k$, $E|_{\mathfrak s_1} \in Der(\mathfrak s_1)$,
    \item $\mathfrak n_1$ is stable under $\mathfrak k$, the set of skew-symmetric derivations,
    \item $\mathfrak a + \mathfrak n_0$ is a subalgebra which is stable under $\mathfrak k$.
\end{enumerate}
\end{lemma}

\begin{proof} We prove (i).  This is a standard result and follows from the fact that any derivation $E$ sends $\mathfrak s$ to $\mathfrak n$.    
If $E$ is skew-symmetric, it must also preserve $\mathfrak a = \mathfrak n^\perp$ and hence vanish on $\mathfrak a$.  The above equation is now just the Jacobi identity on the Lie algebra $Der(\mathfrak s)\ltimes \mathfrak s$.

Part (ii) follows from the facts that derivations preserve the commutator subalgebra and $\mathfrak k$ annihilates $\mathfrak a$.

To prove (iii), observe that $ad~H$ being symmetric yields $\mathfrak n_1 = \sum_{\lambda \not = 0} V_\lambda$, where $V_\lambda$ is the $\lambda$-eigenspace of $ad~H$.  Since $ad~H$ commutes with $\mathfrak k$ (by part (i)), it is clear that $\mathfrak k$ preserves $\mathfrak n_1$.

The last claim follows from the fact that $\mathfrak a$ being abelian implies $ad~\mathfrak a$ preserves the eigenspaces of $ad~H$ and hence 
$\mathfrak n_0$ is $ad~\mathfrak a$ stable.  As above, this subalgebra is $\mathfrak k$-stable as $\mathfrak k$ preserves the eigen spaces of $ad~H$.

\end{proof}

\subsection*{Analysis on the modification $R$ of $S$.}
As stated before, the group $R$ is a so-called modification of $S$.  Thus, there is a linear map $\phi : \mathfrak s \to
\mathfrak k =\mathfrak{so}(\ip{\ ,\ }) \cap Der(\mathfrak s)$ such that $\mathfrak r = (Id +\phi)\mathfrak s$.  
There are minimal conditions on $\phi$, only enough to insure that $R$ will be solvable and act almost simply-transitively.  We record two of these conditions here, see \cite[Prop.~2.4]{GordonWilson:IsomGrpsOfRiemSolv}.

\begin{prop}[Gordon-Wilson]\label{prop: prop 2.4 from GW}  Let $\phi: \mathfrak s \to \mathfrak k$ be a modification map with corresponding modification $r=(Id +\phi)\mathfrak s$.  Then
	\begin{enumerate}
	\item $\phi(\mathfrak s) \subset \mathfrak k $ is abelian,
	\item $[\mathfrak r,\mathfrak r] \subset Ker~\phi$.
	\end{enumerate}
\end{prop}
Consequently, $\mathfrak r < \mathfrak s \rtimes \mathfrak t$, where $\mathfrak t < \mathfrak k$ is a set of commuting, skew-symmetric derivations of $\mathfrak s$, and $R < S\rtimes T$ where $T$ is an abelian group of orthogonal automorphisms of $S$.

\begin{lemma}\label{lemma: modifications of s1} Let $\phi$ be a modification map of a solvsoliton $\mathfrak s$ with $\mathfrak r = (Id+\phi)\mathfrak s$.
    \begin{enumerate}
        \item $[\mathfrak s,\mathfrak s]\subset Ker\ \phi$
        \item Let $H$ be the mean curvature vector of $\mathfrak s$ with image $\widetilde{H} = H+\phi(H) \in \mathfrak r$, then $Ker~ad~\widetilde{H}|_{[\mathfrak s,\mathfrak s]} \subset Ker~ ad~H|_{[\mathfrak s,\mathfrak s]}$.
    \end{enumerate}
\end{lemma}

\begin{remark} In the language of Gordon-Wilson, $(i)$ shows that the modification of solvsoliton is always a  normal modification (cf. \cite[Prop.~2.4]{GordonWilson:IsomGrpsOfRiemSolv}).  For a general solvmanifold, this is not true.
\end{remark}

\begin{proof} 

We prove (i).   
To show $[\mathfrak s,\mathfrak s] \subset Ker~\phi$, it suffices to show $[\mathfrak a, \mathfrak a], [\mathfrak a,\mathfrak{n} ], [\mathfrak{n},\mathfrak{n}] \subset Ker~\phi$, as $\phi$ is a linear map.

As $\mathfrak a$ is abelian, it is trivially true that $[\mathfrak a,\mathfrak a] \subset Ker~\phi$.

Now consider $X\in \mathfrak a$ and $Y\in\mathfrak{n}$.  As $\phi(\mathfrak s)$ is abelian and annihilates $\mathfrak a$ (Prop.~\ref{prop: prop 2.4 from GW} and Lemma \ref{lemma: commuting derivations}), we see that $[\phi(X)+X,\phi(Y)+Y] = \phi(X)Y + [X,Y] \in Ker~\phi$ or, equivalently,
\begin{equation}\label{eqn: D on n lies in Ker phi}
	 ad~(\phi(X) + X) : \mathfrak{n} \to Ker~\phi
\end{equation}
Since every derivation of $\mathfrak s$ takes its image in $\mathfrak{n}$, and $\mathfrak r \subset \mathfrak t \ltimes \mathfrak{s}$, we see that $\mathfrak{n}$ is stable under $B=ad~(\phi(X) + X) = \phi(X) + ad~X$.  Denoting the  eigenspaces of $B$ on $\mathfrak{n}^\mathbb C$ by $V_\lambda$, we have
	$$\mathfrak{n} = \bigoplus  (V_\lambda \oplus  V_{\bar{\lambda} }) \cap \mathfrak{n}$$
As $\phi(X)$ and $ad~X$ commute, each summand is invariant under both $\phi(X)$ and $ad~X$.  

Observe that $Ker~B = Ker~\phi(X) \cap Ker~ad~X$ since $ad~X$ has real eigenvalues while $\phi(X)$ has purely-imaginary eigenvalues and commutes with $ad~X$.  Thus, if $\phi(X)|_{V_\lambda}\not = 0$, then $B$ is non-singular on $V_\lambda$ and $V_{\bar \lambda}$.  This implies $V_\lambda \oplus V_{\bar \lambda}= B(V_\lambda \oplus V_{\bar \lambda})$, which implies
	$$Im(ad~X|_{(V_\lambda \oplus  V_{\bar{\lambda} }) \cap \mathfrak{n} }      )  \subset    (V_\lambda \oplus  V_{\bar{\lambda} }) \cap \mathfrak{n} \subset Im(B|_{\mathfrak{n}}) \subset Ker~\phi$$
The last inclusion follows from Eqn.~\ref{eqn: D on n lies in Ker phi}.

If $\phi(X)|_{V_\lambda}=0$, then $V_\lambda = V_{\bar \lambda}$ and $B|_{V_\lambda} = ad~X|_{V_\lambda}$.  Again, Eqn.~\ref{eqn: D on n lies in Ker phi} yields
	$$Im(ad~X|_{V_\lambda}) \subset Ker~\phi$$
All together, this proves $[\mathfrak a,\mathfrak{n}] \subset Ker~\phi$.

To finish, one must show $[\mathfrak{n},\mathfrak{n}]\subset Ker~\phi$.  However, as every derivation of $\mathfrak s$ preserves $\mathfrak{n}$, we may restrict our modification to $\mathfrak{n}$ and we have a modification
	$$ \mathfrak n' = (id+\phi) \mathfrak{n} \subset \mathfrak t \ltimes \mathfrak{n}$$
of a nilpotent Lie algebra.  Theorem 2.5 of \cite{GordonWilson:IsomGrpsOfRiemSolv} shows that  any modification of a nilpotent subalgebra must be a normal modification.  Now Proposition 2.4 (ii d), loc.~cit., yields $[\mathfrak{n},\mathfrak{n}] \subset Ker~\phi$.   This proves (i).

We prove  (ii).  Recall that $H\in\mathfrak a = \mathfrak n^\perp$ and, from the work above, we have
	$$Ker~ad~\widetilde H = Ker~\phi(H) \cap Ker~ad~H$$
Together with the fact that derivations preserve the commutator subalgebra, we have the desired result.

\end{proof}

\subsection*{The subgroup $C$}
In the following subsection, we study $C$ more closely.  

\begin{lemma}\label{lemma: spliting C into components}  Given $c\in C$, there exist $n\in [S,S]$, $a\in exp(\mathfrak a)$, and $k \in K$ such that $c = nak$.  Moreover, $\log (na)\in \mathfrak a + \mathfrak n_1$.
\end{lemma}

\begin{proof}  Write $c\in C <R$ as $c = s\ \alpha$ where $\alpha \in T < K$  
and $s\in S$.  Then $c\cdot e_S = s$, as $\alpha$ fixes $e_S$.  However, $\log(c\cdot e_S) = \log (s) \in Ker\ D \subset \mathfrak a + \mathfrak n_1 \subset \mathfrak s_1$ (Lemmas \ref{lemma: log cp in Ker D} \& \ref{lemma: Ker D in s_1}) and we see that $s\in S_1$.  Thus, $C$ is a subgroup of $T\ltimes S_1$, as $\mathfrak s_1$ is preserved by $\mathfrak k$ (Lemma \ref{lemma: n1 and s1 stable under K}).  Furthermore, $\mathfrak s_1 = \mathfrak a + [\mathfrak s,\mathfrak s]$ is completely solvable and so we may write $S_1 = exp(\mathfrak a) [S,S]$.  
\end{proof}

\begin{remark}  Consider $na$ in the expression $c=nak$ above.  Although $log(na)\in \mathfrak a + \mathfrak n_1$, we do not know, a priori, that $log(n)\in \mathfrak n_1$ as $\mathfrak n_1$ is not necessarily a subalgebra.  In the sequel, we  show that $n$ is trivial.
\end{remark}

\begin{lemma}\label{lemma: decomposing C}  The group $C$ is a subgroup of $exp(\mathfrak a) T$. 
\end{lemma}

\begin{proof}  Given $c\in C$, write $c=nak$ as in Lemma \ref{lemma: spliting C into components}.  Let $H\in\mathfrak a$ be the mean curvature vector.  Our first step is to show
    $$Ad(n)(H +\phi(H)) = H+\phi(H) $$

By Lemma \ref{lemma: commuting derivations} and the fact that $\mathfrak a$ is abelian, we see that $Ad(c) H = Ad(n)Ad(a)Ad(k)H = Ad(n)H$.  Observing that $Ad(n) = exp( ad\ \log n)$, we obtain
	$$Ad(c)H = Ad(n)H = H + (Ad(n)H)_{[\mathfrak s,\mathfrak s]}$$
where the second term is an element of $[\mathfrak s,\mathfrak s]$. 

 Likewise, as $\mathfrak a$ and $\mathfrak \phi(\mathfrak s)$ commute (Lemma \ref{lemma: commuting derivations}) and $\phi(\mathfrak s)$ is  abelian (Prop.~\ref{prop: prop 2.4 from GW}), we see that $Ad(c)\phi(X) = Ad(n)Ad(a)Ad(k)\phi(H) = Ad(n)\phi(H)$.  Since $[\mathfrak s,\mathfrak s]$ is stable under $T$, $\mathfrak t \ltimes [\mathfrak s,\mathfrak s]$ is a subalgebra of $\mathfrak k\ltimes \mathfrak s$ and  we have $Ad(n)\phi(H) \in \mathfrak t\ltimes \mathfrak [\mathfrak s,\mathfrak s]$.  We write this element  as a sum of  its components 
	$$Ad(c)\phi(H) = Ad(n)\phi(H) = (Ad(n)\phi(H))_\mathfrak t + (Ad(n)\phi(H))_{[\mathfrak s,\mathfrak s]}$$

The Lie algebra $\mathfrak r$ is a subalgebra of $\mathfrak t\ltimes \mathfrak s$ and so $Ad(c)(H+\phi(H))\in \mathfrak r$.  This element can be written as $Y+\phi(Y)$, for some $Y\in \mathfrak s$, and the above work shows that $Y = H + (Ad(n)H)_{[\mathfrak s,\mathfrak s]} + (Ad(n)\phi(H))_{[\mathfrak s,\mathfrak s]}$.  Applying $\phi$ and using Lemma \ref{lemma: modifications of s1}, we see that $\phi(Y) = \phi(H)$.  This says precisely that 
    \begin{equation}\label{eqn: Ad(n) phi X stuff}
    (Ad(n)\phi(H))_\mathfrak t = \phi(H)
    \end{equation}

We now exploit the condition $Ad(c)\in O(\mathfrak r,\ip{\ ,\ }')$.  From Eqn.~\ref{eqn: relating inner products on r and s}, we have
\begin{eqnarray*}
    |H + (Ad(n)H)_{[\mathfrak s,\mathfrak s]}+ (Ad(n)\phi(H))_{[\mathfrak s,\mathfrak s]} | &=& |Y|\\
    		&=& |Y+\phi(Y)|'\\
   			&=&  |Ad(c)(H+\phi(H)) |'\\
           &=& |H+\phi(H) |' \\
           &=& |H|
\end{eqnarray*}
As $[\mathfrak s,\mathfrak s]$ and $\mathfrak a$ are orthogonal relative to $\ip{\ ,\ }$ on $\mathfrak s$, we see that $(Ad(n)H)_{[\mathfrak s,\mathfrak s]} + (Ad(n)\phi(H))_{[\mathfrak s,\mathfrak s]} =0$.  This implies
    \begin{equation}\label{eqn: Ad(n) stuff second time} 
    H + (Ad(n)H)_{[\mathfrak s,\mathfrak s]} + (Ad(n)\phi(H))_{[\mathfrak s,\mathfrak s]} +\phi(H) = H+\phi(H)
    \end{equation}
But $(Ad(n)\phi(H))_\mathfrak t = \phi(H)$, see Eqn.~\ref{eqn: Ad(n) phi X stuff}, and so the left-hand side of Eqn.~\ref{eqn: Ad(n) stuff second time} is precisely $Ad(n)(H+\phi(H))$, i.e.
	$$Ad(n) \widetilde H  = \widetilde H $$
where $\widetilde H = H +\phi(H)$.  This proves our first claim and we continue with the rest of the proof of the lemma.

Upon exponentiating ($\mathfrak t\ltimes \mathfrak s \to K\ltimes S$), we have $n \ exp_R (t \widetilde{H}) \ n^{-1} = exp_R (t \widetilde{H})$, where $exp_R$ denotes the exponential map $\mathfrak r \to R$.  This equality holds for small $t$ and implies
    $$n = exp_R (t \widetilde{H}) \ n \ exp_R (t \widetilde{H})^{-1}$$
Recalling that the exponential map $exp:\mathfrak s \to S$ is a diffeomorphism (as $S$ is completely solvable and simply-connected), we may apply $\log$ to the previous equality to obtain
    $$ \log n = Ad(exp_R \ t\widetilde{H}) \log n$$
where $log|_{[S,S]} : [S,S]\to [\mathfrak s,\mathfrak s]$.  Differentiating at $t=0$ yields $0 = ad\ \widetilde{H} (\log n)$.

Recall that  $Ker~ad\ \widetilde{H}|_{[\mathfrak s,\mathfrak s]} \subset Ker~ad~H|_{[\mathfrak s,\mathfrak s]}$ (Lemma \ref{lemma: modifications of s1}) and so $\log~n \in \mathfrak n_0$.  Thus, we have shown $na$ (appearing in $c=nak$) is an element of the subgroup $exp(\mathfrak a) exp(\mathfrak n_0) <S$  (this is a group by Lemma \ref{lemma: commuting derivations}).  Upon taking the $\log$, we see that $\log (na) \in \mathfrak a + \mathfrak n_0$.

However, Lemma \ref{lemma: spliting C into components} says $\log(na)\in \mathfrak a + \mathfrak n_1$.  As the pairwise intersection of $\mathfrak a$, $\mathfrak n_1$, and $\mathfrak n_0$ is trivial, we obtain $n=e$ and $C<exp(\mathfrak a) T$.

\end{proof}

\begin{lemma} Given $c\in C$, write $c = ak = ka$ as above; i.e., $k\in T<K$ and $a\in exp(\mathfrak a)$.  Then $Ad(c)\in O(\mathfrak s, \ip{\ ,\ })$ and hence $Ad(a)\in O(\mathfrak s, \ip{\ ,\ })$.
\end{lemma}

\begin{proof} The second claim follows immediately from the first since $k\in K$ is an orthogonal automorphism.

To prove the first claim, take $X\in \mathfrak s$ corresponding to $X +\phi(X) \in \mathfrak r$.  Observe that $Ad(c)X\in \mathfrak s$ and $Ad(c) \phi(X) = \phi(X)$, as in the proof of the lemma above.  The condition $Ad(c)\in O(\mathfrak r,\ip{\ ,\ }')$ gives
    $$ |X| = |X+\phi(X)|' = |Ad(c)(X+\phi(X))|' = | Ad(c)X |  $$
As $X$ can be any element of $\mathfrak s$, the first claim is proven.
\end{proof}

We now complete the proof of Proposition \ref{prop: C-orbit has one point}.  Since $\mathfrak s$ is completely solvable, the map $Ad(a): \mathfrak s \to \mathfrak s$ has only positive eigenvalues.  However, for $a$ appearing in $c=ka$, $Ad(a)$ is orthogonal and hence must be the identity map.   This says precisely that $a\in Z(S)$, the center of $S$.  However, $Z(S) < N$, the nilradical of $S$, and hence $a\in N\cap exp(\mathfrak a) = \{ e\}$.  This shows that $C$ is a subgroup of $K$ and hence fixes $e_S\in S$.

This completes the proof of Theorem \ref{thm: Ricci soliton solvmanifolds are isometric to solsolitons}.

\section{Examples of non-(semi-)algebraic Ricci solitons}
\label{sec: examples of non algebraic}

In light of Theorem \ref{thm: Ricci soliton solvmanifolds are isometric to solsolitons},  it is reasonable to ask if every homogeneous Ricci soliton is actually algebraic for any transitive group action.  As  the following examples demonstrate, this is not the case.  To our knowledge, examples of non-algebraic Ricci solitons have not appeared in the literature before.  

\begin{example}\label{ex: non-algebraic soliton} Let $\mathfrak s_1$ be a completely solvable Lie algebra which admits a  non-trivial soliton metric.  Consider $\mathbb R^m$ and an abelian subalgebra $\mathfrak t$ of $\mathfrak{so}(m)$.  Construct the semi-direct product $\mathfrak s_2  = \mathfrak t\ltimes \mathbb R^m$.  The simply-connected Lie group $S$ with Lie algebra $\mathfrak s = \mathfrak s_1 \oplus \mathfrak s_2$ admits a Ricci soliton metric, but this soliton cannot be algebraic relative to the Lie structure of $S$.
\end{example}

The fact that this group cannot admit an algebraic soliton follows immediately from the structural constraints given in Theorem 4.8 (iv) of \cite{Lauret:SolSolitons}.  To see that this space admits  a Ricci soliton metric, first one recognizes that $S_2$ admits a flat metric by a result of Milnor \cite[Thm.~1.5]{Milnor:LeftInvMetricsonLieGroups}.  $S_2$ is now isometric to $\mathbb R^n$, for some $n$. To build the soliton metric on $S$, we take the soliton metric on $S_1$ and direct sum the flat metric on $S_2$. 
Now $S$ is isometric to $S_1 \times \mathbb R^n$. This is a solvsoliton since $c\ Id$ is a derivation of the abelian algebra $Lie \ \mathbb R^n$.

Although Theorem \ref{thm: construction of ricci solitons on solvable} below heavily restricts the algebraic structure on solvable Lie groups admitting Ricci solitons, as we demonstrate, several unexpected phenomena may occur.  

\begin{example}\label{ex: solvable group with Ricci soliton but ad a not reductive} There exists a solvable group $R$ admitting a Ricci soliton metric with the property that $ad\ \mathfrak a$ does not consist of reductive endomorphisms (for any complement $\mathfrak a$ of the nilradical of $\mathfrak r$).
\end{example}

\begin{proof}[Proof of \ref{ex: solvable group with Ricci soliton but ad a not reductive}]
Consider the 5-dimensional Heisenberg Lie algebra $\mathfrak n$ with orthonormal basis $\{X_1,Y_1,X_2,Y_2,Z\}$.  The only non-trivial bracket relations are $[X_1,Y_1]=[X_2,Y_2]=Z$, and the obligatory relations from anti-symmetry of the bracket.  Let $N$ denote the simply-connected nilpotent group with said Lie algebra.

There exists a modification $\mathfrak r$ of $\mathfrak n$ which is solvable and has 4-dimensional nilradical $\mathfrak n(\mathfrak r)$.   
The orthonormal basis of $\mathfrak r$ is $\{\widetilde{X_1},Y_1,X_2,Y_2,Z\}$ where $\widetilde{X_1} = X_1 + \phi(X_1)$ and $\phi(X_1)$ acts on $\mathfrak s$ via $\phi(X_1): \{X_1,Y_1,Z\} \to 0$, $\phi(X_1)X_2 = Y_2$, and $\phi(X_1)Y_2 = -X_2$.  

Notice that any complement to $\mathfrak n(\mathfrak r)$ is one dimensional and spanned by a vector $V$ with non-zero $\widetilde X_1$-component.  Observing that $ad~V$ leaves invariant the span of $\{X_1,Y_1,Z\}$ and is nilpotent on this span, we see that no complement to $\mathfrak n(\mathfrak r)$ consists of reductive endomorphisms.

Let $R$ denote the simply-connected Lie group  with Lie algebra $\mathfrak r$.  As $\mathfrak r$ is a modification of $\mathfrak n$, $R$ acts isometrically on $N$ with its nilsoliton metric.  Now $R$ has a Ricci soliton metric and its Lie algebra has the desired properties.
\end{proof}

\begin{example}\label{ex: solvabe group with Ricci soliton but a is not abelian}  There exists a solvable group admitting a Ricci soliton metric with the property that no complement of the nilradical is abelian.
\end{example}

\begin{proof}
To construct a group with the above properties, will modify a nilpotent Lie group $N$ with nilsoliton metric where $N=N_1 \times N_2$ and each factor has a nilsoliton metric on it.  We choose $N_1$ to be the 3-dimensional Heisenberg group and $N_2$ to be any nilpotent Lie group whose autmorphism group has maximal compact with rank at least 2.  (For example, one could take $N_2$ to be a product of two Heisenberg groups.  We note that among two-step nilpotent Lie groups, most groups satisfy this requirement on the maximal compact of the automorphism group.)

Let ${X,Y,Z}$ be a basis of $\mathfrak n_1$ satisfying $[X,Y]=Z$ and take $\phi_X, \phi_Y \in Der(\mathfrak n_2)$ which are linearly independent, commute, and tangent to compact groups of automorphisms (this is possible by hypothesis).  We define a modification map $\phi$ of $\mathfrak  n = \mathfrak n_1 + \mathfrak n_2$ by
	\begin{eqnarray*}
	X &\to& \phi_X\\
	Y &\to& \phi_Y\\
	Z &\to& 0\\
	\mathfrak n_2 &\to& 0
	\end{eqnarray*} 
The modification $\mathfrak r = (id +\phi)\mathfrak n$ satisfies the hypotheses of Theorem \ref{thm: construction of ricci solitons on solvable}, and so gives a solvable group $R$ with a Ricci soliton metric.

The nilradical of $\mathfrak r$ is spanned by $Z$ and $\mathfrak n_2$.  Any complement will contain vectors $V,W$ such that $V = X+\phi_X + N_V$ and $W=Y+\phi_Y+N_W$, where $N_V,N_W\in \mathfrak n_2$.  Thus,
	$$[V,W] = Z + N_{[V,W]}$$	
for some $N_{[V,W]} \in\mathfrak n_2$.  
We have now constructed a group with the desired property.	
\end{proof}

\begin{remark*}The above phenomena 
are special to Ricci soliton solvmanifolds and cannot happen for Einstein solvmanifolds by the structural results of Lauret, see \cite{LauretStandard}.
\end{remark*}

\section{Structural and Existence questions}
\label{sec: existence questions}

Given the above examples, we ask the following question.

\begin{question} Which solvable Lie groups admit left-invariant Ricci soliton metrics?  
\end{question}

In the special case of solvsolitons, this question is fairly well understood from the work of Lauret \cite{Lauret:SolSolitons}.  There it is shown that every solvsoliton $\mathfrak s$ can be built from a nilsoliton $\mathfrak n$ together with an abelian subalgebra $\mathfrak a \subset Der (\mathfrak n)$ consisting of  reductive elements
	$$\mathfrak s = \mathfrak a + \mathfrak n$$
Below, we show how to construct every solvable Lie group admitting a Ricci soliton from a solvsoliton.

\subsection*{Construction of solvable Lie groups admitting Ricci solitons}  From the work in Section \ref{sec: main result}, we know that every solvable Lie group $R$ admitting a Ricci soliton is a modification of a solvsoliton $S$.  Furthermore, $S$ can be assumed to be completely solvable.  

That is, we begin with a completely solvable Lie algebra $\mathfrak s = Lie~S$  with inner product $\ip{\cdot,\cdot}$ which corresponds to the solvsoliton metric on $S$.  Denote by $\mathfrak k = \mathfrak{so}(\ip{\cdot,\cdot}) \cap Der(\mathfrak s)$ the algebra of skew-symmetric derivations.  Then there exists a  linear map $\phi: \mathfrak s \to \mathfrak k$ such that  $\mathfrak r = (id + \phi) \mathfrak s$.  

One can, of course, start with a linear map $\phi: \mathfrak s \to \mathfrak k$ and try to list the conditions on $\phi$ such that $\mathfrak r = (id + \phi)\mathfrak s$ is a Lie subalgebra of $\mathfrak k\ltimes \mathfrak s$; that is, list all the modifications of $\mathfrak s$.  Although this is difficult to do for a general solvmanifold, it turns out to be simple in the case of $S$ being a solvsoliton.

\begin{thm}\label{thm: construction of ricci solitons on solvable} 
Let $\mathfrak s$ be a completely solvable Lie algebra whose simply-connected Lie group $S$ admits a solvsoliton.  Fix a choice $\mathfrak k$ of maximal compact subalgebra of $Der(\mathfrak s)$ (i.e. the Lie algebra of a maximal compact subgroup of $Aut(S)$).  Consider a linear map $\phi: \mathfrak s \to \mathfrak k$ satisfying 
	\begin{enumerate}
	\item $[\mathfrak s,\mathfrak s]\subset Ker~\phi$
	\item $\phi(\mathfrak s)$ is abelian
	\end{enumerate}
Then $\mathfrak r = (id +\phi)\mathfrak s$ is a solvable Lie algebra whose simply-connected Lie group $R$ admits a Ricci soliton metric.  Conversely, ever solvable Lie group which admits a Ricci soliton metric arises by this construction.
\end{thm}

Before proving this theorem, we present the following lemma.

\begin{lemma}  Let $\mathfrak s$ be a completely solvable algebra whose Lie group $S$ admits a solvsoliton.  Let $\mathfrak k$ be a maximal compact subalgebra of $Der(\mathfrak s)$.  Then there exists a solvsoliton metric $\ip{\cdot, \cdot}$ such that $\mathfrak k = \mathfrak{so}(\ip{\cdot,\cdot}) \cap Der(\mathfrak s)$.   
\end{lemma}

\begin{proof}
	This lemma follows immediately from the proof of \cite[Theorem 4.1]{Jablo:ConceringExistenceOfEinstein}.  To apply that work, first pick an inner product $\ip{\cdot, \cdot}'$ such that $\mathfrak k \subset \mathfrak{so}(\ip{\cdot,\cdot}') \cap Der(\mathfrak s)$.  (As maximal compact groups are all conjugate, the maximality of $\mathfrak k$ implies this inclusion is an equality.)
	
	In the statement of \cite[Theorem 4.1]{Jablo:ConceringExistenceOfEinstein}, the groups considered are unimodular, completely solvable, and admit solvsoliton metrics.  However, unimodularity is only used to understand the isometry groups of those solvsolitons.  The result we claim is proven in Step 1, on the middle of page 747, loc. cit.  In the notation of that work, it is shown
	$$    Aut(\mu) \cap O(\ip{\cdot,\cdot}) \subset  Aut(\mu) \cap O(   (g^{-1})^* \ip{\cdot,\cdot})   $$
where the first inner product is generic and the second is a solvsoliton.  Thus we have a soliton metric $\ip{\cdot,\cdot}$ such that 
	$$   \mathfrak k = \mathfrak{so}(\ip{\cdot,\cdot}') \cap Der(\mathfrak s) \subset \mathfrak{so}(\ip{\cdot,\cdot}) \cap Der(\mathfrak s)  $$
By maximality, we see that $\mathfrak k = \mathfrak{so}(\ip{\cdot,\cdot}) \cap Der(\mathfrak s)$.
\end{proof}

\begin{proof}[Proof of Theorem \ref{thm: construction of ricci solitons on solvable}]
We first remark that the choice of maximal compact subgroup of $Aut(S)$ does not matter as maximal compact subgroups are all conjugate, hence the constructions using one maximal compact will be isomorphic to the constructions using any other.  By the lemma above, we may assume throughout that $\mathfrak s$ is endowed with the solvsoliton metric $\ip{\cdot, \cdot}$ such that $\mathfrak k = \mathfrak{so}(\ip{\cdot,\cdot}) \cap Der(\mathfrak s)$.

Next, we show that  conditions (i) and (ii) ensure that $\mathfrak r$ is a Lie subalgebra of $\mathfrak k\ltimes \mathfrak s$; i.e. $\mathfrak r$ is closed under the bracket of $\mathfrak k\ltimes \mathfrak s$.  By (ii), for $X,Y\in\mathfrak s$, we have 
	$$[X + \phi(X) , Y + \phi(Y)] = [X,Y] + \phi(X)Y - \phi(Y)X$$
Furthermore, (i) implies  $[X + \phi(X) , Y + \phi(Y)]  \in \mathfrak r$ if and only if $\phi(X)Y - \phi(Y)X \in\mathfrak r$.    Using  \cite[Prop.~2.4]{GordonWilson:IsomGrpsOfRiemSolv}, we see that (i) is equivalent to $\phi(X)Y\in Ker~\phi = \mathfrak r \cap \mathfrak s$, for all $X,Y\in\mathfrak s$.  This shows $\mathfrak r$ is a Lie subalgebra of $\mathfrak k\ltimes \mathfrak s$.  
That the Lie algebra   $\mathfrak r$ is solvable follows from the observation that $[\mathfrak r,\mathfrak r]\subset \mathfrak s$.

Let $R$ denote the simply-connected Lie group with Lie algebra $\mathfrak r$.  By \cite[Theorem 3.1]{GordonWilson:IsomGrpsOfRiemSolv}, $R$ acts transitively on $S$ endowed with its solvsoliton metric.  As such, $R$ inherits a left-invariant metric which is Ricci soliton.
\medskip

We now prove the converse.  Let $R$ be a solvable Lie group with Ricci soliton metric.  By Theorem \ref{thm: Ricci soliton solvmanifolds are isometric to solsolitons}, $R$ is isometric to a solvsoliton $S$ which we can assume is completely solvable.  As $R$ acts transitively on $S$, and completely solvable groups are always in standard position (\cite[Theorem 4.3]{GordonWilson:IsomGrpsOfRiemSolv}), the group $R$ is a modification of $S$.  That is, there exists a map $\phi : \mathfrak s \to \mathfrak k$ such that $\mathfrak r = (id + \phi) \mathfrak s$.  By Lemma \ref{lemma: modifications of s1}, we see that (i) holds.  Condition (ii) holds as $\mathfrak r$ is a solvable Lie algebra.
\end{proof}

\subsection*{Characterization of solvsolitons}   Theorem \ref{thm: construction of ricci solitons on solvable} shows how to obtain all solvable Lie groups admitting Ricci soliton metrics by modifying a completely solvable Lie group admitting a solvsoliton.  We now characterize those modifications which are solvsolitons.  In the following,  $\mathfrak n(\mathfrak r)$ denotes the nilradical of $\mathfrak r$.

\begin{prop}\label{cor: R solvsoliton if and only if n in Ker phi}  Let $R$ be a solvable Lie group endowed with a Ricci soliton metric.  The following are equivalent.
	\begin{enumerate}
	\item $R$ is a solvsoliton.
	\item Describing $R$ as a modification with $\mathfrak r = (id + \phi) \mathfrak s$ (as in Theorem \ref{thm: construction of ricci solitons on solvable}), we have $\mathfrak n  \subset Ker~\phi$, where $\mathfrak n$ is the nilradical of $\mathfrak s$.
	\item The elements of $ad~\mathfrak n (\mathfrak r)^\perp$ are reductive with eigenvalues which are not all purely imaginary.
	\item The set $\mathfrak n(\mathfrak r)^\perp$ is an abelian subalgebra of $\mathfrak r$ which is $ad$-reductive and whose elements have eigenvalues which are not all purely imaginary.
\end{enumerate}
\end{prop}

\begin{remark}  We have included (iv) above to demonstrate that one can determine when a Ricci soliton on a solvable group is a solvsoliton using only a finite amount of information; namely, one can check (iv) on a basis of $\mathfrak n(\mathfrak r)^\perp$. \end{remark}

This proposition gives a characterization of algebraic (and hence semi-algebraic solitons, cf. Lemma \ref{lemma: evolving by aut implies algebraic}) on solvable Lie groups.   In general, it seems to be a very difficult problem to characterize when a homogeneous Ricci soliton $(M,g)$ is algebraic or semi-algebraic  relative  to a given $G\subset Isom(M,g)$.  It is not even known if being semi-algebraic implies being algebraic, outside the class of solvmanifolds.  For more on this problem, we direct the interested reader to the  work \cite{LauretLafuente:StructureOfHomogeneousRicciSolitonsAndTheAlekseevskiiConjecture}.

Before starting the  proof of the proposition, we recall  from the work of Lauret \cite[Theorem 4.8]{Lauret:SolSolitons} 
that we know  every solvsoliton $\mathfrak s$ may be decomposed as an orthogonal direct sum
	$$\mathfrak s = \mathfrak a + \mathfrak n$$
where $\mathfrak n$ is the nilradical of $\mathfrak s$ and $\mathfrak a$ is an abelian.

\begin{proof}[Proof of Proposition \ref{cor: R solvsoliton if and only if n in Ker phi}]
The implications
	\begin{quote}
	(ii) implies (i) implies (iv)
	\end{quote}
follow immediatly from \cite[Theorem 4.8]{Lauret:SolSolitons}.  Clearly (iv) implies (iii).  

We finish by showing (iii) implies (ii).  As the modification $\mathfrak r$ of $\mathfrak s$ is a normal modification (Lemma \ref{lemma: modifications of s1}), we see that   $\mathfrak n(\mathfrak r) \subset \mathfrak n \cap Ker~\phi \subset \mathfrak n$ and so $\mathfrak a = \mathfrak n^\perp \subset \mathfrak n(\mathfrak r)^\perp$.  Thus $\mathfrak n(\mathfrak r)^\perp = \mathfrak a \oplus (\mathfrak n(\mathfrak r)^\perp \cap \mathfrak n)$.  (Here we are computing $\mathfrak n(\mathfrak r)^\perp$ as a subset of $\mathfrak s$ using the solvsoliton metric on $\mathfrak s$.)

Take $X\in \mathfrak n(\mathfrak r)^\perp \cap \mathfrak n$.  As our modification is normal, the derivation $\phi(X)\in Der(\mathfrak s)$ is also a derivation of $\mathfrak r$.  Derivations of solvable Lie algebras take their image the nilradical and using the skew-symmetry of $\phi(X)$ we see that $\phi(X) X = 0$.  Thus $[ad~X,\phi(X)]=0$.  Now we see that the eigenvalues of $ad~X+\phi(X)$ are the sums of eigenvalues of $\phi(X)$ and $ad~X$.  As the later are zero, $ad~X+\phi(X)$ has only purely imaginary and by hypothesis $X=0$.  That is, $\mathfrak n(\mathfrak r)^\perp =\mathfrak a$ and so $\mathfrak n = \mathfrak n(\mathfrak r) \subset Ker~\phi$, as desired.

\end{proof}

\bibliographystyle{amsalpha}

\providecommand{\bysame}{\leavevmode\hbox to3em{\hrulefill}\thinspace}
\providecommand{\MR}{\relax\ifhmode\unskip\space\fi MR }
\providecommand{\MRhref}[2]{%
  \href{http://www.ams.org/mathscinet-getitem?mr=#1}{#2}
}
\providecommand{\href}[2]{#2}

\end{document}